\documentclass[11pt]{article}

\usepackage{amsmath,amsfonts,amssymb,amsthm,bbm}
\usepackage{graphics,epsfig}
\usepackage{color}
\usepackage{array}
\usepackage{ulem}
\usepackage{chngpage}
\newtheorem{theorem}{Theorem}[section]
\newtheorem{proposition}{Proposition}[section]
\newtheorem{lemma}{Lemma}[section]

\numberwithin{equation}{section}

\allowdisplaybreaks

\renewcommand\today{\number\year/\number\month/\number\day}

\def\qed{\hfill$\Box$\medskip}

\begin{document}

\noindent\makebox[60mm][1]{\tt {\large Version:~\today}}

\bigskip
\noindent{
{\Large\bf  Scaling limit of the local time of the $(1,L)-$random
walk}\footnote{
\noindent  Supported by NSFC (NO.11131003) and  985 Program.
}
\\

\noindent{
Wenming Hong\footnote{ School of Mathematical Sciences
\& Laboratory of Mathematics and Complex Systems, Beijing Normal
University, Beijing 100875, P.R. China. Email: wmhong@bnu.edu.cn} ~  Hui Yang\footnote{ School of Mathematical Sciences
\& Laboratory of Mathematics and Complex Systems, Beijing Normal
University, Beijing 100875, P.R. China. Email: yanghui2011@mail.bnu.edu.cn}

\noindent{
(Beijing Normal University)
}

}

\vspace{0.1 true cm}

\begin{center}
\begin{minipage}[c]{12cm}
\begin{center}\textbf{Abstract}\end{center}
\bigskip

It is well known (Donsker's Invariance Principle) that the random walk converges to Brownian motion by scaling.
 In this paper, we
will prove that the scaled local time of the $(1,L)-$random walk
converges to that of the Brownian motion. The results was proved by
Rogers (1984) in the case $L=1$. Our proof is based on the intrinsic multiple branching structure within the $(1,L)-$random walk revealed by Hong and Wang ( 2013).\\

\mbox{}\textbf{Keywords:}\quad random walk, multi-type branching process, local time, Brownian motion.\\
\mbox{}\textbf{Mathematics Subject Classification}:  Primary 60J80;
secondary 60G50.
\end{minipage}
\end{center}


\section{ Introduction and Main Results\label{s1}}

 Donsker's Invariance Principle tells us that the  random walk converge to the
Brownian motion by proper space and time scaling. It is naturally to consider the scaling limit of  the  local times. One can not get it directly from the continuous theorem because the local time is not a continuous function of the Brownian motion. For the simple random walk (i.e., $L=1$, the nearest random walk), Rogers (\cite{R},1984) confirmed the result based on the branching structure within the simple symmetric random walk first introduced by Dwass (\cite{D75}, 1975) and the convergence of the scaling branching processes demonstrated by  Lamperti (\cite{La67}, 1967) and Lindvall (\cite{Lin72},1972), combining the Ray-Knight Theorem. ~In the present paper we consider  the   $(1,L)-$random walk. At first we can express the local time as a linear function of the intrinsic multi-type branching processes within the $(1,L)-$random walk, which has been revealed  by Hong and Wang (\cite{HW}, 2013) recently. After that, We obtain  the scaling limit following the usual schedule by proving the convergence of the finite dimensional distribution and the tightness.

We consider the $(1,L)-$ random walk on the half line reflected at 0 , i.e., a Markov chain $ \{X_n\}_{n\geq 0}$ on $\mathbb{Z^{+}}=\{0, 1, 2, \ldots \}$ with $X_0=0$ and the transition probabilities specified by,
for $n\geq 0$,~$i\geq 1$,
\begin{flalign}\label{a}
 P(X_{n+1}=1|X_{n}=0)& =1\nonumber\\
P(X_{n+1}=i+l|X_{n}=i)& =
\begin{cases}
p_{l}, & \mbox{for $l=1,\cdots, L $}, \\
q, & \mbox{for $l=-1$}.
\end{cases}
\end{flalign}
where $p_{1}+p_{2}+\cdots+p_L+q=1$, $ 0<p_{1},p_{2},\cdots,p_L,q<1$ and ``symmetric"
\begin{equation}\label{00}
EX_{n}=p_{1}+2p_{2}+\cdots+ L\cdot  p_L-q=0.
\end{equation}  Obviously, this Markov chain  is irreducible and recurrent. For any position $j\geq 0$, the {\it local time} $L(j; n)$ at $j$   is defined as the visiting number by the Markov chain $\{X_n\}_{n\geq 0}$ before   $n$,
\begin{flalign}\label{lt}
L(j; n)&=\#\{0\leq r\leq n:X_{r}=j\} \quad \text{
    for } j,n \geq 0.
\end{flalign}
Define the excursion time at position $0$, $\tau_{0}=0$, and for $n\geq 1$,
 \begin{equation}\label{e}
\begin{split}
\tau_{n}=\inf\{k>\tau_{n-1}:X_{k}=0\}.
\end{split}
\end{equation}
We are interested in the local time $L(j; \tau_{N})$, the visiting number at position $j$ by $\{X_n\}_{n\geq 0}$ in the first $N$ excursions.  Define the scaling $l_{N}(x)$ as follows, for $\forall N\in Z^{+},x\in [0,1]$,
\begin{equation}\label{sl}
l_{N}(x)=
\begin{cases}
\frac{L([Nx];  ~\tau_{N})}{N}, & \mbox{for $Nx\geq1$}, \\
2/\sigma^{2}, & \mbox{for $0\leq Nx <1$},
\end{cases}
\end{equation}
where $\sigma^2=DX_n=p_1+q+4p_2=6q-2$, $n>1$.
Our main results is the following
\begin{theorem}\label{main} As the random elements on $D[0,\infty)$,
\begin{equation}\label{1}
l_{N}(\cdot)\Rightarrow H(\cdot)
\end{equation}
where $H(x),x\in[0,\infty)$ is a diffusion processes which is the solution of the stochastic differential equation
\begin{flalign}\label{2}
&H(x)=\frac{2}{\sigma^{2}}+\frac{2}{\sigma}\int_{0}^{x}(H(s)^{+})^{\frac{1}{2}}dB_{s},
\end{flalign}
where $B(t),t\in[0,\infty)$ is standard Brownian motion.
\end{theorem}

\noindent {\bf Remark 1.1} {\it Actually $(H({t}))_{t\geq0}$ is a continuous (time and states) branching process with $H(0)=\frac{2}{\sigma^{2}}$ and its
transition probabilities satisfy, for $\lambda\geqslant0$,
\begin{equation}\label{3}
\int e^{-\lambda y}p_t(x,dy)=exp\left(-x\frac{\lambda}{1+\frac{2}{\sigma^2}\lambda t}\right).
\end{equation}
We will prove the Theorem following the usual schedule by proving the convergence of the finite dimensional distribution and the tightness.\qed
}

\noindent {\bf Remark 1.2} {\it  Let $\{B(t): t\geq0\}$ be Brownian motion on $\mathbf{R}, B_{0}=0$, and let $l(x,t)$ be its local time. Define $T=\inf\{t:l(0,t)>1\}$. Ray-Knight Theorem tells us that
\begin{equation*}
(l(x,T))_{x\geq0}=(Z_{t})_{t\geq0},
\end{equation*}
where $Z$ is the solution of the stochastic differential equation
\begin{flalign*}
&Z_{t}=1+\sqrt{2}\int_{0}^{t}(Z_{s}^{+})^{\frac{1}{2}}dB_{s},
\end{flalign*}
and its
transition probabilities satisfy, for $\lambda\geqslant0$,
\begin{displaymath}
\int e^{-\lambda y}p_t(x,dy)=e^{-x\frac{\lambda}{1+\lambda t}}.
\end{displaymath}
From this point of view, (\ref{1}) can be rewritten as
    \begin{equation}\label{0}
l_{N}(\cdot)\Rightarrow \frac{1}{\sigma^{2}}l^{*}(\cdot,T)
\end{equation}
where $l^{*}(x,T)=l(x,T)+l(-x,T), x\geq 0,$ which  coincides with the scaling behavior of the $(1,L)$-random walk $\frac{X_{[nx]}}{\sqrt{n}}\longrightarrow B_{\sigma^2x}$.   \qed
}


\section{ Local time and branching process in the $(1,L)-$ random walk  \label{s2}}

We will  express the local time in terms of the intrinsic branching structure within the $(1,L)-$ random walk revealed by Hong and Wang (\cite{HW}, 2013).  For simplicity of the notation, in what follows, we will restrict ourselves to consider the case $L=2$.

\subsection{Branching process within the $(1,L)-$random walk}
Let us recall the 2-type branching processes within the $(1,2)-$ random walk(\cite{HW}, 2013).
For $i\geq 1$, to record the visiting number at the position $i$ by the walk $\{X_n\}_{n\geq 0}$ in the first excursion at $0$, define
\begin{flalign}
\begin{split}
    U_{1}(i-1)&=\#\{0\leq n<\tau_{1}:X_{n}< i,~X_{n+1}=i\},\label{du1}\\
    U_{2}(i-1)&=\#\{0\leq n<\tau_{1}:X_{n}< i,~X_{n+1}=i+1\}.
\end{split}
\end{flalign}
For $n\geq 0$, let $ U_{n}=(U_{1}(n),U_{2}(n)).$
\begin{theorem}\label{bp}\mbox{(Hong $\&$ Wang, \cite{HW}, 2013)} (1)The process $\{U_n\}_{n=0}^{\infty}$ is a $2$-type critical branching process whose branching mechanism is given by $P(U_{0}=(1,0))=1$,
and for $k\geq 1$
\begin{flalign}
\begin{split}
&P^{(1)}(u_{1},u_{2}):=P(U_{k+1}=(u_{1},u_{2})\big|U_{k}=e_{1})=\frac{(u_{1}+u_{2})!}{u_{1}!u_{2}!}p_{1}^{u_{1}}p_{2}^{u_{2}}q,\label{u1}\\
&P^{(2)}(u_{1},u_{2}):=P(U_{k+1}=(u_{1}+1,u_{2})\big|U_{k}=e_{2})=\frac{(u_{1}+u_{2})!}{u_{1}!u_{2}!}p_{1}^{u_{1}}p_{2}^{u_{2}}q.
\end{split}
\end{flalign}
(2)Let $m_{ij}=E(U_j(n+1)|U_n=e_i) $ be  the mean offspring  of type j particles born  from a single type i parent particle. Define the mean offspring matrix $M=\{m_{ij},i,j=1,2\}$,
then
\begin{flalign}
{M} =\left(\begin{array}{ccccc}
\rho_{1} & \rho_{2} \\
1+\rho_{1} & \rho_{2}
\end{array} \right)\label{m}
\end{flalign}
where $\rho_{i}=\frac{p_{i}}{q},i=1,2$.
\end{theorem}
\noindent{\bf Remark 2.1.} This  is  followed from the result of branching structure in the $(L,1)$ random walk which is revealed by Hong and  Wang(\cite{HW}, 2013).  A little bit attention should be noted is here we view the branching structure from the ``upward" direction whereas
  in Hong and  Wang(\cite{HW}, 2013) from the ``downward" direction, because here we consider the reflected random walk.
   We should  consider here the branching processes $\{U_n\}_{n\geq 0}$ begin at $n\geq 1$ and with $U_0=(1,0)$ as the
   ``immigration". In addition, the ``symmetric" condition (\ref{00}) ensures the maximal eigenvalue of the offspring
   matrix $M$ is 1, i.e., the multitype branching process  $\{U_n\}_{n\geq 0}$ is critical.

\

With Theorem \ref{bp} in hand, we can calculate the probability generating function of $\{U_n\}_{n\geq 0}$ as follows, which is useful in the proof of the scaling limit.

\begin{proposition} (1)Denote $\boldsymbol{s}=(s_{1},s_{2})$, $g^{(i)}(s_{1},s_{2}):=E(\boldsymbol{s}^{U_2}|U_1=e_i)$, $i=1,2;$ then
\begin{equation}
\begin{split}
g^{(1)}(s_{1},s_{2})=\frac{q}{1-p_{1}s_{1}-p_{2}s_{2}}\label{g1}\\
g^{(2)}(s_{1},s_{2})=\frac{qs_{1}}{1-p_{1}s_{1}-p_{2}s_{2}}
\end{split}
\end{equation}
(2)Let $f_n(s_{1},s_{2})$ be the generating function of $\{U_n=(U_{1}(n),U_{2}(n))\}_{n=0}^{\infty}$, i.e., $f_n(s_{1},s_{2}):=E(\boldsymbol{s}^{U_n}|U_0=e_1)$, we have for $n\geq 1$
\begin{equation}\label{fn}
f_{n}(s_{1},s_{2})=\frac{1+(1-s_{1})a_{n-1}+(1-s_{2})b_{n-1}}{1+(1-s_{1})a_{n}+(1-s_{2})b_{n}}
\end{equation}
where $(a_{0},b_{0})=(0,0)$, for $n>0$, $(a_{n},b_{n})=\boldsymbol{u}(M^{n-1}+M^{n-2}+ \cdots +M+I)$  and $\boldsymbol{u} = (\rho_{1},\rho_{2}).$
\end{proposition}
\noindent {\it Proof.} (1) By direct calculation from the branching mechanism (\ref{u1}),
\begin{equation*}
\begin{split}
g^{(1)}(s_{1},s_{2})&=\sum^{\infty}_{u_{1},u_{2}=0}P^{(1)}(u_{1},u_{2})s_{1}^{u_{1}}s_{2}^{u_{2}}\\
&=\sum^{\infty}_{u_{1},u_{2}=0}\frac{(u_{1}+u_{2})!}{u_{1}!u_{2}!}p_{1}^{u_{1}}p_{2}^{u_{2}}q s_{1}^{u_{1}}s_{2}^{u_{2}}\\
&=\frac{q}{1-p_{1}s_{1}-p_{2}s_{2}}\\
&=\frac{1}{1+(1-s_{1})\rho_{1}+(1-s_{2})\rho_{2}},\\
\end{split}
\end{equation*}
and similarly to get $g^{(2)}(s_{1},s_{2})$.

\noindent(2) We will show (\ref{fn}) by induction. Firstly, when $n=1$,
\begin{flalign*}
f_{1}(\boldsymbol{s})=g^{(1)}(s_{1},s_{2})& =\frac{1}{1+(1-s_{1})\rho_{1}+(1-s_{2})\rho_{2}}\\
&=\frac{1+(1-s_{1})a_{0}+(1-s_{2})b_{0}}{1+(1-s_{1})a_{1}+(1-s_{2})b_{1}}.
\end{flalign*}
Assume (\ref{fn}) is true when $k\leq n-1$,  we have
\begin{flalign*}
f_{n}(\boldsymbol{s})&=f_{n-1}(g^{(1)}(s_{1},s_{2})),g^{(2)}(s_{1},s_{2})\\
&=\frac{1+(1-\frac{1}{1+(1-s_{1})\rho_{1}+(1-s_{2})\rho_{2}})a_{n-2}+(1-\frac{s_{1}}{1+(1-s_{1})\rho_{1}+(1-s_{2})\rho_{2}})b_{n-2}}{1+(1-\frac{1}{1+(1-s_{1})\rho_{1}+(1-s_{2})\rho_{2}})a_{n-1}+(1-\frac{s_{1}}{1+(1-s_{1})\rho_{1}+(1-s_{2})\rho_{2}})b_{n-1}}\\
&=\frac{1+(1-s_{1})(\rho_{1}+\rho_{1}a_{n-2}+\rho_{1}b_{n-2}+b_{n-2})+(1-s_{2})(\rho_{2}+\rho_{2}a_{n-2}+\rho_{2}b_{n-2})}{1+(1-s_{1})(\rho_{1}+\rho_{1}a_{n-1}+\rho_{1}b_{n-1}+b_{n-1})+(1-s_{2})(\rho_{2}+\rho_{2}a_{n-1}+\rho_{2}b_{n-1})}\\
&=\frac{[(a_{n-2},b_{n-2})M+(\rho_{1},\rho_{2})](1-s_{1},1-s_{2})^{\prime}}{[(a_{n-1},b_{n-1})M+(\rho_{1},\rho_{2})](1-s_{1},1-s_{2})^{\prime}}\\
&=\frac{1+(1-s_{1})a_{n-1}+(1-s_{2})b_{n-1}}{1+(1-s_{1})a_{n}+(1-s_{2})b_{n}}
\end{flalign*}
and $(a_{n},b_{n})=(a_{n-1},b_{n-1})M+(\rho_{1},\rho_{2})$ for $n\geq 1$. \qed

\noindent{\bf Remark 2.2}~  {\it It should be better to write $f^{(1)}_n(s_{1},s_{2}) $ for $f_n(s_{1},s_{2})$ because $f_n(s_{1},s_{2}):=E(\boldsymbol{s}^{U_n}|U_0=\boldsymbol{e_1})$. In next section, we need the other one $f^{(2)}_n(s_{1},s_{2}):=E(\boldsymbol{s}^{U_n}|U_0=\boldsymbol{e_2})$. By the similar calculation, we have $f^{(2)}_{1}(s_{1},s_{2})=g^{(2)}(s_{1},s_{2})$, and for $n\geq 2$,
\begin{equation}\label{fn2}
f^{(2)}_{n}(s_{1},s_{2})=\frac{1+(1-s_{1})a_{n-2}+(1-s_{2})b_{n-2}}{1+(1-s_{1})a_{n}+(1-s_{2})b_{n}}.
\end{equation} \qed

}

\subsection{Local time $L(j; \tau_{N})$}

From the definition of $ U_{n}=(U_{1}(n),U_{2}(n))$ in (\ref{du1}), we can easily express the local time $L(j; \tau_{N})$ in terms of the $2$-type branching processes $\{U_n\}_{n\geq 0}$ as follows,

\begin{theorem}\label{lc} (1) For $j\geq 1$,
\begin{equation}\label{ex1}
 L(j; \tau_{1}) = U_{1}(j-1)+U_{1}(j)+U_{2}(j).
\end{equation}
(2) For any positive integral $N$,
\begin{equation}\label{exn}
 L(j; \tau_{N}) = \sum_{r=1}^{N} \xi_r,
\end{equation}
where $\xi_r$, $r=1,2,\cdots$ are i.i.d. random variables, distributed as $L(j; \tau_{1})$.
\end{theorem}
\noindent{\it Proof} (1)  The local time $L(j; \tau_{1})$ is the visiting number at position $j$ by the  trajectory of the $(1,2)-$random walk within the first excursion at $0$ (i.e., between $0\leq n\leq \tau$),  it  is the summation of two kind steps: ``upper  steps"  (visits at position $j$ from below $j$) and ``down  steps"  (visits at position $j$ from above $j$). By the definition (\ref{du1}), the ``upper  steps" is just  $U_{1}(j-1)$; with regard the recurrence of the $(1,2)$-walk and the walk downs step by step, the  ``down  steps" to $j$ equals to the  steps from (and below) $j$ to $j+1$ (which is $U_{1}(j)$) plus the steps from $j$ to $j+2$ (which is $U_{2}(j)$); and so (\ref{ex1}) is followed.

\noindent (2) Decompose the trajectory of the $(1,2)$-random walk, $L(j; \tau_{N})$ is  the summation of the visiting numbers at $j$ in $N$ excursions, which are the i.i.d. random variables. For more details, write $L(j; m,n):=\{m\leq r\leq n:X_{r}=j\}$, we have
\begin{flalign}\label{ltn}
L(j; \tau_{N}) = \sum_{r=1}^{N} L(j; \tau_{r-1}, \tau_{r}).
\end{flalign}
Write $\xi_r:=L(j; \tau_{r-1}, \tau_{r})$, then $\xi_r$, $r=1,2,\cdots$ are i.i.d. random variables, distributed as $L(j; \tau_{1})$  by the Markov property of the $(1,2)$-random walk.
\qed

\section{Proof of  Theorem \ref{main}  \label{s2}}
With the explicit expression of the local time (\ref{exn}) in terms of the muti-type branching process $\{U_n\}_{n\geq 0}$,  we are now at the position to prove the main result.
Firstly, note that as in the (1) of Theorem \ref{lc}, we have
\begin{flalign}\label{xi}
\xi_r:=L(j; \tau_{r-1}, \tau_{r}) = U^{(r)}_{1}(j-1)+U^{(r)}_{1}(j)+U^{(r)}_{2}(j),
\end{flalign}
where $ \{ U^{(r)}_{n}=(U^{(r)}_{1}(n),U^{(r)}_{2}(n)); n\geq 0\}$, $r=1,2,\cdots$ are i.i.d., distributed as $\{ (U_{1}(n), U_{2}(n) );   n\geq 0\}$. Actually, for each $r\geq 1$, $ \{ U^{(r)}_{n}; n\geq 0\}$ is a $2$-type branching processes corresponding the $r^{th}$ excursion at $0$ of the random walk, which is independent of each other  and with the same branching mechanism with $ \{ U_{n}=(U_{1}(n), U_{2}(n) ); n\geq 0\}$. Recall (\ref{sl}) the scaling of the local time $L(j; \tau_{N})$, by (\ref{ltn}) and (\ref{xi}),
\begin{equation}\label{sl1}
\begin{split}
l_{N}(x)&=\frac{L([Nx];\tau_{N})}{N} =\sum^{N}_{r=1}{\frac{L([Nx]; \tau_{r-1}, \tau_{r})}{N}} \\
&=\sum^{N}_{r=1}{\frac{U^{(r)}_{1}([Nx]-1)+U^{(r)}_{1}([Nx])+U^{(r)}_{2}([Nx])}{N}} \\
&:=U_{N,1}(x-\frac{1}{N})+U_{N,1}(x)+U_{N,2}(x),
\end{split}
\end{equation}
where we write,  for $i=1 ,2$,
\begin{equation}\label{sl1}
U_{N,i}(x)=\sum^{N}_{r=1}{\frac{U^{(r)}_{i}([Nx])}{N}}.
\end{equation}

{So in what follows, we just need to consider the weak convergence of $\{2U_{N,1}(x)+U_{N,2}(x); x\geq 0\}$, with regard of the strong convergence  to $0$ of
$U_{N,1}(x-\frac{1}{N})-U_{N,1}(x)$ by the strong law of large numbers.  Nakagawa (\cite{N86}, 1986) considered the convergence of critical multitype Galton-Watson branching processes, which generalized the results of Lindvall (\cite{Lin72}, 1972) to the multitype case. However, here we can not apply the result directly, because our target $\{2U_{N,1}(x)+U_{N,2}(x); x\geq 0\}$ is different from the $\{\widehat{Y}_n(t); t\geq 0\}$ in Theorem 1.1 of (\cite{N86}, 1986). We can calculate explicitly  based on the branching structure Theorem  \ref{bp}  to specify the role of the $\sigma^2$ in (\ref{2}) and (\ref{3}) comparing with Theorem 1.1 of Nakagawa (\cite{N86}, 1986).

\

\noindent{\it Step 1 Warm up: the convergence of one dimensional distribution}
\begin{lemma}\label{3.1}
For $ x\in[0,1]$, as $N\to\infty$
\begin{equation*}
2U_{N,1}(x)+U_{N,2}(x) \Rightarrow  H(x),
\end{equation*}
where the laplace transform of $H(x)$ is given by
\begin{flalign*}
\Phi(x,\lambda)=exp\left(-\frac{\frac{2}{\sigma^2}\lambda}{1+\frac{2}{\sigma^2}x\lambda}\right),
\end{flalign*}
and $\sigma^2=DX_n=p_1+q+4p_2=6q-2$, $n>1$.
\end{lemma}
\noindent{\it Proof}~  Recall the notation (\ref{sl1}), we calculate the
  Laplace transformation of  $2U_{N,1}(x)+U_{N,2}(x)$,
\begin{flalign}\label{lf}
F_{2U_{N,1}(x)+U_{N,2}(x)}(\lambda)&=E exp \left[-\lambda(2U_{N,1}(x)+U_{N,2}(x)|U_0=\boldsymbol{e}_1\right]\nonumber\\
&=\left\{E exp \left[-\frac{\lambda}{N}\left[(2U_{1}([Nx])+U_{2}([Nx])\right]\right]\right\}^N\nonumber\\
&=\left[f_{[Nx]}(e^{-\frac{2\lambda}{N}}, e^{-\frac{\lambda}{N}})\right]^N\nonumber\\
&=\left[\frac{1+(a_{[Nx]-1},b_{[Nx]-1})\boldsymbol{v}'}{1+(a_{[Nx]},b_{[Nx]})\boldsymbol{v}'}\right]^N.
\end{flalign}
the last step is from (\ref{fn}), where $\boldsymbol{v}=(1-e^{-\frac{2\lambda}{N}},1-e^{-\frac{\lambda}{N}})$, $(a_{n},b_{n})=\boldsymbol{u}(M^{n-1}+M^{n-2}+ \cdots +M+I)$  and $\boldsymbol{u} = (\rho_{1},\rho_{2}).$
If we denote
\begin{equation}\label{AB}
 \begin{cases}
A_N(x): = (a_{[Nx]-1},b_{[Nx]-1})\boldsymbol{v} ', &  \\
B_N(x): = (a_{[Nx] },b_{[Nx] })\boldsymbol{v} '.
\end{cases}
\end{equation}
(\ref{lf}) can be written as
\begin{flalign}\label{lf1}
F_{2U_{N,1}(x)+U_{N,2}(x)}(\lambda)&=\left(\frac{1+A_N(x)}{1+B_N(x)}\right)^N\nonumber\\
&=\left[\left(1+\frac{A_N(x)-B_N(x)}{1+B_N(x)}\right)^{(1+B_N(x))/(A_N(x)-B_N(x))}\right]^{N(A_N(x)-B_N(x))/(1+B_N(x))}.
\end{flalign}
We are now to consider the limit of $A_N(x), B_N(x)$ and $N(A_N(x)-B_N(x))$. To this end, recall
$M$ is the mean offspring matrix , see (\ref{m}).  Recall that in our ``symmetric"  model, we have $p_1+p_2+q=1$, and $EX_n=p_1+2p_2-q=0$, $n\geq 1$. By calculation, two of the eigenvalues of $M$ are $1$ and $\alpha=\frac{1-2q}{q}$ (and $|\alpha|<1$); and there is a  matrix $T$
\begin{flalign}
 T=\left(\begin{array}{ccccc}
1 & 1-2q \\
2 & 1-q
\end{array} \right),~~~  &T^{-1}=\frac{1}{3q-1}\left(\begin{array}{ccccc}
1-q &  2q-1 \\
-2 & 1
\end{array} \right) \label{dT}
\end{flalign}
such that
\begin{flalign}
{M} =T\left(\begin{array}{ccccc}
1 & 0 \\
0 & \alpha
\end{array} \right)T^{-1}. \label{Mn}
\end{flalign}
Then
\begin{eqnarray}\label{An}
& &A_N(x): = (a_{[Nx]-1},b_{[Nx]-1})\boldsymbol{v} ' =\boldsymbol{u}(M^{[Nx]-2}+M^{[Nx]-3}+ \cdots +M+I)\boldsymbol{v}'\nonumber\\
& &=\boldsymbol{u}T \left(\begin{array}{ccccc}
[Nx]-1& 0 \\
0 & \alpha^{[Nx]-2}+\cdots +\alpha +1
\end{array} \right) T^{-1}\boldsymbol{v}'\nonumber\\
& &= \frac{1}{3q-1}\left(1, ~\frac{(2q-1)^2}{q}\right) \left(\begin{array}{ccccc}
[Nx]-1& 0 \\
0 & \alpha^{[Nx]-2}+\cdots +\alpha +1
\end{array} \right) \left(\begin{array}{ccccc}
((q-1)e^{-\frac{\lambda}{N}}-q)( e^{-\frac{\lambda}{N}}-1) \\
(2e^{-\frac{\lambda}{N}}+1)(e^{-\frac{\lambda}{N}}-1)
\end{array} \right) \nonumber\\
& &\longrightarrow  \lambda x \frac{2}{\sigma^2},
\end{eqnarray}
as $N\to\infty$. Similarly, we get
\begin{eqnarray}\label{Bn}
B_N(x): = (a_{[Nx] },b_{[Nx] })\boldsymbol{v} ' =\boldsymbol{u}(M^{[Nx]-1}+M^{[Nx]-2}+ \cdots +M+I)\boldsymbol{v}' \longrightarrow   \lambda x \frac{2}{\sigma^2}.
\end{eqnarray}
Obviously, $B_N(x)-A_N(x)= \boldsymbol{u}M^{[Nx]-1}\boldsymbol{v}'$, and
\begin{eqnarray}\label{ABN}
N(B_N(x)-A_N(x))=N( \boldsymbol{u}M^{[Nx]-1}\boldsymbol{v}' )\longrightarrow   \lambda  \frac{2}{\sigma^2}.
\end{eqnarray}
Combining (\ref{An})-(\ref{ABN}) with (\ref{lf1}), we get
\begin{flalign*}
\lim_{N\rightarrow \infty}F_{2U_{N,1}(x)+U_{N,2}(x)}(\lambda)=exp\left(-\frac{\frac{2}{\sigma^2}\lambda}{1+\frac{2}{\sigma^2}x\lambda}\right)
\end{flalign*}
complete the proof. \qed

\noindent{\bf Remark 3.1}~  {\it We write $F^{(1)}_{N}(x; \lambda):=F_{2U_{N,1}(x)+U_{N,2}(x)}(\lambda)=E exp \left[-\lambda(2U_{N,1}(x)+U_{N,2}(x)|U_0=\boldsymbol{e}_1\right]$. In next step, we need the other one $F^{(2)}_{N}(x; \lambda):=E exp \left[-\lambda(2U_{N,1}(x)+U_{N,2}(x)|U_0=\boldsymbol{e}_2\right]$. By the similar calculation  (with regard of (\ref{fn2})), we can get
\begin{flalign}\label{F12}
\lim_{N\rightarrow \infty}F^{(2)}_{N}(x; \lambda)=\lim_{N\rightarrow \infty}F^{(1)}_{N}(x; \lambda)=\Phi(x,\lambda).
\end{flalign}}
\qed

\noindent{\it Step 2 ~ The convergence of finite dimensional distributions}

\begin{lemma}\label{3.2} For $1 \leq i \leq k,  \lambda_{i} \geq 0$;  and $  0 \leq x_{1} \leq x_{2} \leq \cdots \leq x_{k}\leq1$, we have, as $N\rightarrow \infty$
\begin{flalign}\label{32}
&E\left(exp\left\{-\sum^{k}_{i=1}\lambda_{i}(2U_{N,1}(x_{i})+U_{N,2}(x_{i}))\right\}\right)\rightarrow E\left(exp\left\{-\sum^{k}_{i=1}\lambda_{i}H(x_{i})\right\}\right),
\end{flalign}
where $\{H(x),x\in[0,1]$\}   is a diffusion, continuous (time and states) branching process with $H(0)=\frac{2}{\sigma^{2}}$ given in (\ref{2}).
\end{lemma}
\begin{proof} First of all, note that $\{H(x),x\in[0,1]$\}   is a  continuous (time and states) branching process whose Laplace transforms satisfy
\begin{equation}\label{33}
E\left(exp\left\{-\sum^{k}_{i=1}\lambda_{i}H(x_{i})\right\}\right)=E\left(exp\left\{-\sum^{k-1}_{i=1}\lambda_{i}H(x_{i})\right\}\Phi^{H(x_{k-1})}
(x_{k}-x_{k-1},\lambda_{k})\right).
\end{equation}
 We will prove (\ref{32}) by induction. Firstly, Lemma (\ref{3.1}) convince (\ref{32}) for $k=1$.
 Assume (\ref{32}) is right when $k\leq m$, we will check (\ref{32}) for   $k=m+1$.
\begin{flalign*}
&E(exp\{-\sum^{m+1}_{i=1}\lambda_{i}(2U_{N,1}(x_{i})+U_{N,2}(x_{i}))\})\\
&=E(E(exp\{-\sum^{m+1}_{i=1}\lambda_{i}(2U_{N,1}(x_{i})+U_{N,2}(x_{i}))\}|U_{N}(x_{m})))\\
&=E\left(exp\{-\sum^{m}_{i=1}\lambda_{i}(2U_{N,1}(x_{i})+U_{N,2}(x_{i}))\}
E(exp\{-\lambda_{m+1}(2U_{N,1}(x_{m+1})+U_{N,2}(x_{m+1}))\}|U_{N}(x_{m}))\right),
\end{flalign*}
in which with the notation $F^{(1)}_{N}(x; \lambda)$ and $F^{(2)}_{N}(x; \lambda)$ in Remark 3.1,
 \begin{flalign*}
&E(exp\{-\lambda_{m+1}(2U_{N,1}(x_{m+1})+U_{N,2}(x_{m+1}))\}|U_{N}(x_{m})))\\
&=[F_{N}^{(1)}(x_{m+1}-x_{m},\lambda_{m+1})]^{2U_{N,1}(x_{m})}
[F_{N}^{(2)}(x_{m+1}-x_{m},\lambda_{m+1})]^{U_{N,2}(x_{m})}.
\end{flalign*}
then, (\ref{F12}) and the induction  for $k\leq m$ enable us to conclude that
\begin{flalign*}
&E(exp\{-\sum^{m+1}_{i=1}\lambda_{i}(2U_{N,1}(x_{i})+U_{N,2}(x_{i}))\})\\
&\longrightarrow E[exp(-\sum_{i=1}^{m}\lambda_{i}H(i))\Phi^{H(x_{m})}(x_{m+1}-x_{m},\lambda_{m+1})],
\end{flalign*}
which is (\ref{32}) for $k=m+1$ with regard of (\ref{33}).
\end{proof}

\noindent{\it Step 3 ~ completing the proof of Theorem \ref{main}}

\

We follow the standard schedule. The  tool for proving the weak convergence of  $\{2U_{N,1}(x)+U_{N,2}(x)\}~ to ~\{H(x)\}$ is Theorem 13.5 in Billingsley (\cite{B99}, 1999) which states that if
$\{V_{n}\}_{n\geq1}$, $V$ are random elements in $D[0,1]$ and
\begin{adjustwidth}{1em}{0em}
(a) the finite dimensional distributions of $V_{n}$ converge to those of V;\\
(b) the probability for jumps of $V$ in the point 1 is zero;\\
(c) there exist $\gamma\geq0,~\alpha>1$ and $F$ is continuous and non-decreasing on $[0,1]$, such that, $E[|V_{n}(t_{2})-V_{n}(t)|^{\gamma}\cdot|V_{n}(t)-V_{n}(t_{1})|^{\gamma}]\leq|F(t_{2})-F(t_{1})|^{\alpha}$, holds for all $n$ and all $0\leq t_{1}\leq t\leq t_{2}\leq1$;
\end{adjustwidth}
then $V_{n}~\Rightarrow ~V$ on $D[0,1]$.\\
Part (a) is already established for $\{2U_{N,1}+U_{N,2}\}_{N\geq1}$ in step 2. Since the processes $H$ has continuous paths, (b) is also no problem. We just need to manage (c). To this end, we can follow  Nakagawa (\cite{N86}, 1986) almost line by line with some modifications as following.

 From (\ref{bp}) we know $\{U_{n}\}$ is critical branching process satisfying the condition in (\cite{N86}, 1986), so the second moments of $U_{n}$ has similar asymptotic behavior, i.e.
\begin{equation*}
d_{ij}(n)=E[U_{i}(n)U_{j}(n)]=(U_{0}\cdot \boldsymbol{\mu})Q_{2}[\boldsymbol{\mu}]\nu_{i}\nu_{j}n+o(n)~~ as ~~n\rightarrow \infty, (i,j=1,2)
\end{equation*}
where\begin{equation*}
Q_{2}[\boldsymbol{\mu}]=\boldsymbol{\nu}\cdot \boldsymbol{q_{2}[\boldsymbol{\mu}]}~and~(\boldsymbol{q_{2}[\mu]})_{i}=\sum_{j=1}^{2}\sum_{k=1}^{2}\mu_{j}b_{jk}^{(i)}\mu_{k},~~~~i=1,2,
\end{equation*}
with $\boldsymbol{\mu}=(\mu_{1},\mu_{2})^{\prime},\boldsymbol{\nu}=(\nu_{1},\nu_{2})$ is the right and left eigenvectors corresponding to the eigenvalue 1 of M in (\ref{m}), and $\boldsymbol{\nu}\cdot\boldsymbol{\mu}=1, ~\boldsymbol{1}\cdot \boldsymbol{\mu}=1$. $b^{(i)}_{jk}=E\{U_{j}(1)U_{k}(1)|U(0)=\boldsymbol{e_{i}}\}-E\{U_{j}(1)|U_{0}=
\boldsymbol{e_{i}}\}E\{U_{k}(1)|U_{0}=\boldsymbol{e_{i}}\}$. For our model, as (3.3) in Nakagawa(\cite{N86}, 1986), one has
\begin{equation}
E[(2U_{1}(n)+U_{2}(2))^{2}]=K_{1}(U_{0}\cdot \boldsymbol{\mu})Q_{2}[\boldsymbol{\mu}]n+o(n) ~~ as ~~n\rightarrow\infty
\end{equation}
with $K_{1}=(2\nu_{1}+\nu_{2})^2$; and as a consequence, corresponding (3.5) in Nakagawa(\cite{N86}, 1986) we get,
for arbitrary integers $m>n\geq0$
\begin{flalign}\label{se}
&E[(U_{m}\cdot \boldsymbol{a}-U_{n}\cdot \boldsymbol{a})^{2}|U_{n}]\nonumber\\
&=E[(U_{m}\cdot\boldsymbol{a})^{2}|U_{n}]+(U_{n}\cdot\boldsymbol{a})^{2}-2(U_{n}\cdot\boldsymbol{a})(E[U_{m}\cdot\boldsymbol{a}|U_{n}])\nonumber\\
&=K_{1}(U_{n}\cdot\boldsymbol{\mu})Q_{2}[\boldsymbol{\mu}](m-n)+o(m-n)+(U_{n}\cdot\boldsymbol{a})^{2}-2(U_{n}\cdot\boldsymbol{a})(U_{n}\cdot M^{m-n}\cdot\boldsymbol{a})\nonumber\\
&\leq K_{1}(U_{n}\cdot\boldsymbol{\mu})Q_{2}[\boldsymbol{\mu}](m-n),
\end{flalign}
the last inequality holds because by calculation  $2M^{m-n}\cdot\boldsymbol{a}>\boldsymbol{a}$. This is enough to get (c) as Nakagawa(\cite{N86}, 1986), and the proof is finished for the convergence in the $D[0,1]$. It is easy to extend the convergence in $D[0,N]$ for any positive integer $N$ by the scaling property of the local time of Brownian motion, and then in $D[0,\infty).$   \qed

\noindent{\large{\bf Acknowledgements}} The authors would like to
thank Dr. Hongyan Sun and Ke Zhou for their stimulating discussions.  This project is partially supported by the
National Nature Science Foundation of China (Grant No.\,11131003 ) and 985 project.

\end {document}